\documentclass[a4paper,12pt,reqno]{amsart}
\usepackage{amsfonts}
\usepackage{mathrsfs}
\usepackage{amsmath,amssymb,latexsym,amscd, yfonts}
\usepackage{cases}
\usepackage{enumitem}
\usepackage[all]{xy}
\allowdisplaybreaks[4]

\newtheorem{thm}{Theorem}[]

\newtheorem{cor}[thm]{Corollary}
\newtheorem{prop}[thm]{Proposition}
\newtheorem{claim}[thm]{Claim}

\theoremstyle{definition}

\theoremstyle{remark}

\numberwithin{equation}{section}

\newcommand{\bC}{{\mathbb C}}
\newcommand{\bQ}{{\mathbb Q}}
\newcommand{\bP}{{\mathbb P}}

\newcommand\Supp{{\text{\rm Supp}}}
\newcommand\lct{{\text{\rm lct}}}
\newcommand\glct{{\text{\rm glct}}}
\newcommand\coeff{{\text{\rm coeff}}}

\newcommand\vol{\text{\rm vol}}
\newcommand\Mov{{\text{\rm Mov}}}

\title{Addendum to ``The Noether inequality for algebraic threefolds"}
\author{Jungkai A. Chen, Meng Chen, Chen Jiang}

\address{\rm National Center for Theoretical Sciences, Taipei Office, and Department of Mathematics, National Taiwan University, Taipei, 106, Taiwan}
\email{jkchen@ntu.edu.tw}

\address{\rm School of Mathematical Sciences \& LMNS, Fudan University,
Shanghai 200433, People's Republic of China}
\email{mchen@fudan.edu.cn}

\address{\rm Shanghai Center for Mathematical Sciences, Fudan University, Jiangwan Campus, 2005 Songhu Road, Shanghai, 200438, People's Republic of China}
\email{chenjiang@fudan.edu.cn}

\thanks{The first author was partially supported by
National Science Council of Taiwan. The second author was
supported by National Natural Science Foundation of China (\#11571076, \#11731004). The third author was supported by Start-up Grant No. SXH1414010.}

\begin{document}
\numberwithin{equation}{section}

\maketitle

\pagestyle{myheadings} \markboth{\hfill J. A. Chen, M. Chen, and C. Jiang
\hfill}{\hfill The Noether inequality for algebraic $3$-folds\hfill}

In our previous work \cite{CCJ}, we proved a Noether-type inequality for projective $3$-fold of general type. The main theorem is \cite[Theorem 1.1]{CCJ} which states that if $p_g(X) \le 4$ or $p_g(X) \ge 21$, then $\vol(X)\geq \frac{4}{3}p_g(X)-{\frac{10}{3}}$. 

The purpose of this note is to prove the following improvement: 

\begin{thm}\label{new main} Let $X$ be a projective $3$-fold of general type and either $p_g(X)\leq 4$ or $p_g(X)\geq 11$. Then 
\begin{equation*}\vol(X)\geq \frac{4}{3}p_g(X)-{\frac{10}{3}}.
\end{equation*}

\end{thm}

Throughout we work over complex number $\bC$. We refer to \cite{KM} for standard concepts and definitions. We keep the same notation as in \cite{CCJ}.

The improvement that the previous condition ``$p_g(X)\geq 21$'' in \cite[Theorem 1.1]{CCJ} is relaxed to ``$p_g(X)\geq 11$'' is due to the following modification of \cite[Proposition 3.1]{CCJ}.

\begin{prop}\label{new pencil bound}
Let $X$ be a minimal projective $3$-fold of general type. Assume that there exists a resolution $\pi:W\to X$ such that $W$ admits a fibration structure $f: W\to \bP^1$. Denote by $F$ a general fiber of $f$ and $F_0$ the minimal model of $F$. Assume that
\begin{enumerate}
 \item there exists a $\pi$-exceptional prime divisor $E_0$ on $W$ such that $(\pi^*K_X|_F\cdot E_0|_F)>0$;
 \item $\pi^*K_X\sim_\bQ bF+D$ for some rational number $b$ and an effective $\bQ$-divisor $D$ on $W$.
 \end{enumerate}
 Then $b< \frac{1}{\glct(F_0)}$ where $\glct(F_0)$ is the global log canonical threshold of $F_0$ (see \cite[Definition 2.6]{CCJ}).
 \end{prop}
 
In \cite[Proposition 3.1]{CCJ}, it was proved  that $ b< \frac{2}{\glct(F_0)} \le 20$ if $F_0$ is a $(1,2)$-surface. Now it can be improved to $b <10$.  By using Proposition \ref{new pencil bound} in the proof of \cite[Corollary 3.3(2)]{CCJ}, it follows that \cite[Corollary 3.3]{CCJ} can be improved as the following:

\begin{cor}\label{new base12} Let $X$ be a minimal projective $3$-fold of general type such that
$|K_X|$ is composed with a pencil of $(1,2)$-surfaces. Assume that one of the following holds:
\begin{enumerate}
\item $|K_X|$ is composed with an irrational pencil; or
\item $|K_X|$ is composed with a rational pencil and $p_g(X)\geq 11$; or
\item $X$ is Gorenstein.
\end{enumerate}
Then there exists a minimal projective $3$-fold $Y$, being birational to $X$, such that $\Mov |K_{Y}|$ is base point free.
\end{cor}

Finally, replace \cite[Corollary 3.3]{CCJ} by Corollary \ref{new base12} in the proof of \cite[Theorem 1.2(3)]{CCJ}. One sees the asserted improvement. Therefore, it is sufficient to prove the Proposition  \ref{new pencil bound}.

\begin{proof}[Proof of Proposition 2.]
Note that, according to the projection formula, the assumptions in the proposition still hold if we replace $W$ with any higher birational model over $W$ (that is, a smooth variety $W'$ with a proper birational morphism $W'\to W$) and replace $E_0$ with its proper transform.
Take $g:W_0\to \bP^1$ to be a relative minimal model of $f: W\to \bP^1$ (for the definition, see \cite[Definition 3.50]{KM}), of which the general fiber is $F_0$. Modulo a further birational modification, we may assume that $f$ factors through $g$ by a morphism $\zeta: W\to W_0$ and the support of $\pi$-exceptional divisors is simple normal crossing.
We may write
$$K_W=\pi^*K_X+E_\pi,$$
where $E_\pi=\sum_{E_i \text{: $\pi$-exc}} a_i E_i$ is an effective $\pi$-exceptional $\bQ$-divisor with $a_i>0$ for all $i$. Being a minimal model of $W$, $X$ is also a minimal model of $W_0$ and we may write
$$
\zeta^*K_{W_0}=\pi^*K_X+\hat{E},
$$
where $\hat{E}=\sum_{E_i \text{: $\pi$-exc}} b_i E_i$ is an effective $\pi$-exceptional $\bQ$-divisor.

Take a general fiber $F$ of $f$, by the assumption,
there exists a  $\pi$-exceptional prime divisor $E_0$ on $W$ such that $(\pi^*K_X|_F\cdot E_0|_F)>0$.
By the projection formula, this means that $(K_X\cdot \pi_*(E_0|_F))>0$.
In particular, $E_0|_F$ is not contracted by $\pi.$ Hence there exists a curve $\Gamma_X\subseteq X$ such that $(K_X\cdot \Gamma_X)>0$ and that
$$\Gamma_X\subseteq \pi(E_0\cap F)\subseteq \pi(E_0).$$ On the other hand, since  $\pi(E_0)$ is a subvariety of codimension at least $2$, we see $\Gamma_X= \pi(E_0)$. In particular, $\Gamma_X$ is independent of $F$, and for any general fiber $F$ of $f$, $\Gamma_X= \pi(E_0\cap F)$.

We claim the following.

\begin{claim}
There exists a $\pi$-exceptional divisor $E_1$ on $W$ such that for a general fiber $F$ of $f$,
\begin{enumerate}
\item $\pi(E_1)=\pi(E_1\cap F)=\Gamma_X$;
\item $\coeff_{E_1}(\pi^*\pi_*F-F-E_\pi)\geq 0$;

\item $\coeff_{E_1}(\pi^*\pi_*F-F)\geq 1$;
\item $\coeff_{E_1}(\hat{E})\geq 1$.
\end{enumerate}
\end{claim}

\begin{proof} We may write $$\pi^*\pi_*F-F =\sum_{E_i \text{: $\pi$-exc}} c_i E_i$$ for some $c_i \ge 0$.
We consider $$\Delta: = \pi^*\pi_*F-F-E_\pi=\sum_{E_i \text{: $\pi$-exc}} (c_i-a_i) E_i .$$

Since
$-\big(K_W+F+\Delta \big)=-\pi^*(K_X+\pi_*F)$,
which is $\pi$-nef and $\pi_*\big(F+\Delta \big)=\pi_*F$ is effective, by the Connectedness Lemma (see \cite[Theorem 5.48]{KM}),$$
\text{Supp}(F +\lfloor \Delta^{\geq 0} \rfloor)\cap \pi^{-1}(x)
$$
is connected for any $x \in X$. Here $\Delta^{\geq 0}$ denotes the part of $\Delta$ with non-negative coefficients.

Denote by $$G:=\sum_{\pi(E_i)=\Gamma_X, c_i-a_i \ge 1} (c_i-a_i) E_i  \le \Delta .$$
We thus conclude that 
$$
\text{Supp} (F + G)\cap \pi^{-1}(x)
$$
is connected for any general point $x\in \Gamma_X$.

Firstly, suppose that $G\neq 0$, then by construction, there  exists a prime divisor, say $E_1\subset \text{Supp} (G)$, on $W$ such that $$\coeff_{E_1}(\pi^*\pi_*F-F-E_\pi)=c_1-a_1 \geq 1$$ and that $E_1$ intersects $F$ 
along $\pi^{-1}(x)$ for any general point $x\in \Gamma_X$. The latter one implies that 
$\pi(E_1\cap F)=\Gamma_X$. Hence $E_1$ satisfies $(1)\sim (3)$.
  

Now suppose that $G= 0$, then there is no exceptional divisor over  $\Gamma_X$ with discrepancy smaller than $-1$,  which implies that   $(X, \pi_*F)$ is plt (i.e. purely log terminal) in a neighborhood of $\eta_{\Gamma_X}$, where $\eta_{\Gamma_X}$ is the generic point of $\Gamma_X$ (see \cite[Definition 2.34]{KM} for the definition of plt). In particular, $\pi_*F$ is normal in a neighborhood of $\eta_{\Gamma_X}$ by \cite[Proposition 5.51]{KM}. Hence $\pi|_F:F\to \pi(F)$ is an isomorphism in a neighborhood of $\eta_{\Gamma_X}$ since it is a point of codimension $1$ in $\pi(F)$. On the other hand, $\eta_{\Gamma_X}$ is a smooth point on $X$ since $X$ is terminal (\cite[Corollary 5.18]{KM}), so by adjunction we have
$$ K_F+ \Delta|_F= \big( \pi^*(K_X+\pi_*F)\big) |_F =(\pi|_F)^* \big((K_X+\pi_*F)|_{\pi(F)} \big)
=(\pi|_F)^*K_{\pi(F)}
$$
 in a neighborhood of $\eta_{\Gamma_X}$.
 This implies that $\Delta|_F=0$ over a neighborhood of $\eta_{\Gamma_X}$.
In particular, since $\pi(E_0)=\pi(E_0\cap F)=\Gamma_X$, we have $\coeff_{E_0}(\Delta)=0$.
 We may take  $E_1=E_0$ which satisfies the requirement $(1)\sim(2)$ of the claim.
 The third statement follows from (2) and the fact  that  $X$ has terminal singularities and hence smooth at $\eta_{\Gamma_X}$, which implies that $\coeff_{E_1}E_\pi$ is a positive integer.

In the end we prove the forth statement. Note that $\zeta(E_1)$ is of dimension at least 1 since it intersects general fibers of $g$, also note that $W_0$ has terminal singularities, hence $W_0$ is smooth  at the generic point of $\zeta(E_1)$. Also recall that $X$ is smooth at $\eta_{\Gamma_X}$, hence $\coeff_{E_1}\hat{E}$ is a non-negative integer. So it suffices to show that  $E_1\subseteq \Supp(\hat{E}).$ Assume, to the contrary, that $E_1\not\subseteq \Supp(\hat{E})$. Then as  $F$ is a general fiber of $f$, $E_1|_F$ has no common component with $\Supp(\hat{E}|_F)$. Since $\Gamma_X=\pi(E_1\cap F)$, we can find a curve $\Gamma_W\subseteq E_1\cap F$ such that $\Gamma_X=\pi(\Gamma_W)$ and $\Gamma_W\not\subseteq \Supp(\hat{E}|_F)$.
Recall that $(K_X\cdot \Gamma_X)>0$.
Hence
\begin{align*}
(\zeta_F^*K_{F_0}\cdot \Gamma_W){}&=(\zeta^*K_{W_0}|_F\cdot \Gamma_W)\\
{}&= ((\pi^*K_{X}|_F+\hat{E}|_F)\cdot \Gamma_W)\\
{}&\geq (\pi^*K_{X}|_F\cdot \Gamma_W)\\
{}&=(K_X\cdot \pi_*\Gamma_W)>0,
\end{align*}
here $\zeta_F=\zeta|_F:F\to F_0$.
In particular, $\Gamma_W$ is not contracted by $\zeta_F$ and hence $E_1$ is not contracted by $\zeta$ as $F$ is general. This means that $E_1$ is a divisor exceptional over $X$ but not exceptional over $W_0$, and therefore $\coeff_{E_1}\hat{E}>0$ as $X$ is a minimal model of $W_0$ (see \cite[Definition 3.50($5^m$)]{KM})), a contradiction.
\end{proof}

Now go back to the proof of the proposition. Take $w=1/b$. Denote
$$
\Delta_W:=-E_\pi+\pi^*\pi_*(F+wD)+(1+w)\hat{E}.
$$
Note that $F+wD \sim \pi^* wK_X$ and hence 
$ F+wD= \pi^* \pi_* (  F +  wD)$. 
Then
$$ K_W+\Delta_W\sim_\bQ (1+w)\pi^*K_X+(1+w)\hat{E}\sim_\bQ(1+w)\zeta^*K_{W_0}.$$
We may write
$$ K_W+\Delta_W=\zeta^*(K_{W_0}+\Delta_{W_0}),$$
where $\Delta_{W_0}=\zeta_*\Delta_{W}\sim_\bQ wK_{W_0}$. Also note that $$\Delta_{W_0}=\zeta_*\Delta_{W}=\zeta_*(\Delta_{W}+E_\pi-\hat{E})\geq 0$$ since $E_\pi-\hat{E}=K_W-\zeta^*K_{W_0}$ is $\zeta$-exceptional.
Restricting on a general fiber $F$ of $f$, we have
$$ K_F+\Delta_W|_F=\zeta_F^*(K_{F_0}+\Delta_{W_0}|_{F_0}),$$
where $\zeta_F=\zeta|_F:F\to F_0$.
By the construction of $E_1$,
\begin{align*}
\coeff_{E_1}\Delta_W\geq (1+w)\coeff_{E_1}\hat{E}>1
\end{align*}
and $E_1\cap F\neq \emptyset$, hence $\Delta_W|_F$ contains a component with coefficient $> 1$.
This implies that $(F_0, \Delta_{W_0}|_{F_0})$ is not lc. On the other hand, $\Delta_{W_0}|_{F_0}\sim_\bQ wK_{F_0}$,
hence
$$
\frac{1}{b}=w> \lct\Big(F_0;\frac{1}{w}\Delta_{W_0}|_{F_0}\Big)\geq \glct(F_0).
$$
\end{proof}

\end{document}